\documentclass[a4paper,11pt,english]{amsart}
\usepackage{amsmath}
\usepackage{amsfonts}
\usepackage{amssymb}
\usepackage{amsthm}
\usepackage[all]{xy}
\usepackage{geometry}
\usepackage{babel}
\newtheorem {thm}{Theorem}
\newtheorem* {thm*}{Theorem}

\newtheorem {cor}[thm]{Corollary}
\newtheorem* {cor*}{Corollary}
\newtheorem {lem}[thm]{Lemma}

\newtheorem {rem}[thm]{Remark}

\theoremstyle{definition}
\newtheorem {exa}[thm]{Example}
\newtheorem* {conj*}{Conjecture}

\DeclareMathOperator{\End}{End}

\DeclareMathOperator{\ord}{ord}
\DeclareMathOperator{\Gal}{Gal}

\DeclareMathOperator{\id}{id}
\DeclareMathOperator{\Spec}{Spec}
\DeclareMathOperator{\Frob}{Fr}

\DeclareMathOperator{\tors}{tors}
\DeclareMathOperator{\card}{\#}

\newcommand{\lord}{\ord_\ell}

\author{Antonella Perucca}
\title{On the reduction of points on abelian varieties and tori}
\date{}

\begin{document}
\maketitle

\begin{abstract}
Let $G$ be the product of an abelian variety and a torus defined over a number field $K$. Let $R_1,\ldots, R_n$ be points in $G(K)$. Let $\ell$ be a rational prime and let $a_1,\ldots, a_n$ be non-negative integers. Consider the set of primes $\mathfrak p$ of $K$ satisfying the following condition: the $\ell$-adic valuation of the order of $(R_i \bmod \mathfrak p)$ equals $a_i$ for every $i=1,\ldots,n$.
We show that this set has a natural density and we characterize the $n$-tuples $a_1,\ldots, a_n$ for which the density is positive.
More generally, we study the $\ell$-part of the reduction of the points.
\end{abstract}

\section{Introduction}

Let $G$ be the product of an abelian variety and a torus defined over a number field $K$. Let $\mathcal O$ be the ring of integers of $K$.
We reduce $G$ modulo $\mathfrak p$, where $\mathfrak p$ is a prime of $K$ (a non-zero prime ideal of $\mathcal O$). By fixing a model of $G$ over an open subscheme of $\Spec \mathcal O$, one can define the reduction $G_{\mathfrak p}$ of $G$ for all but finitely many primes $\mathfrak p$ of $K$.
We fix a point $R$ in $G(K)$ and consider its reduction $(R \bmod \mathfrak p)$, which is well-defined for all but finitely many primes $\mathfrak p$ of $K$ (the set of excluded primes depends on the point, unless the toric part of $G$ is trivial). We are interested in the set of values taken by the order of $(R \bmod \mathfrak p)$, by varying $\mathfrak p$.

If $R$ is a torsion point of order $n$ then the order of $(R \bmod \mathfrak p)$ equals $n$ for all but finitely many primes $\mathfrak p$ of $K$: the excluded primes are either of bad reduction or divide $n$ (bad reduction here means that the reduction is not defined on $R$ or that $G_{\mathfrak p}$ is not the product of an abelian variety and a torus). 

Now assume that $R$ has infinite order. Call $n_R$ the number of connected components of the smallest $K$-algebraic subgroup of $G$ containing $R$.
In \cite[Main Theorem]{Peruccaorder} we proved that $n_R$ is the greatest positive integer dividing the order of $(R \bmod \mathfrak p)$ for all but finitely many primes $\mathfrak p$ of $K$.

Let $\ell$ be a rational prime. We study the $\ell$-adic valuation of the order of $(R \bmod \mathfrak p)$. We write $\lord$ to indicate the $\ell$-adic valuation of the order.
Let $a$ be a non-negative integer and consider the following set:
$$\Gamma=\{\mathfrak p :\; \; \lord(R \bmod \mathfrak p)=a\}$$
We prove that $\Gamma$ is finite if $a< v_\ell(n_R)$ and it has a positive natural density if $a\geq v_\ell(n_R)$. See Corollary~\ref{onepoint}.

For several points we have the following result:

\begin{thm}\label{exmainl}
Let $K$ be a number field, let $I=\{1,\ldots, n\}$. For every $i \in I$, let $G_i$ be the product of an abelian variety and a torus defined over $K$ and let $R_i$ be a point in $G_i(K)$. Let $\ell$ be a rational prime. For every $i\in I$, let $a_{i}$ be a non-negative integer. Consider the following set of primes of $K$:
$$\Gamma=\{\mathfrak p :\; \forall i\in I\;\; \lord(R_i \bmod \mathfrak p)=a_{i}\}$$
The set $\Gamma$ is either finite or it has a positive natural density.

Write $G=\prod_{i=1}^n G_i$ and $R=(R_1,\ldots, R_n)$. Let $G_R$ be the smallest $K$-algebraic subgroup of $G$ containing $R$ and call $G^1_R$ the connected component of $G_R$ containing $R$.

The set $\Gamma$ is infinite if and only if the following condition is satisfied: there exists a torsion point $T=(T_{1},\ldots, T_{n})$ in $G^1_R(\bar{K})$ such that $\lord T_{i}=a_{i}$ for every $i \in I$.
\end{thm}

Let $G$ be the product of an abelian variety and a torus defined over a number field $K$. Let $R$ be a point in $G(K)$. Let $\ell$ be a rational prime and let $\mathfrak p$ be a prime of $K$ of good reduction, not over $\ell$. Call $a=\lord (R\bmod \mathfrak p)$.
 Let $L$ be a finite Galois extension of $K$ where the points in $G[\ell^a]$ are defined. Then for every prime $\mathfrak q$ of $L$ over $\mathfrak p$ there exists a unique $T$ in $G[\ell^a]$ such that $\lord(R-T \bmod \mathfrak q)=0$.
We define the $\ell$-part of $(R \bmod \mathfrak p)$ as the $\Gal(\bar{K}/K)$-class of $T$, which is independent of the choice of $\mathfrak q$ and of $L$.

\begin{thm}\label{thmintro}
Let $G$ be the product of an abelian variety and a torus defined over $K$. Let $R$ be a point in $G(K)$. Let $\ell$ be a rational prime. Let $L$ be a finite Galois extension of $K$. Let $\mathcal T$ be a  $\Gal(\bar{K}/K)$-stable subset of $G[\ell^\infty](L)$.
Then the following set of primes of $K$ is either finite or it has a positive natural density:
$$\Gamma=\{\mathfrak p : \; \textit{$\forall$ prime $\mathfrak q$ of $L$ over $\mathfrak p$}\;\; \ord_\ell(R-Y \bmod \mathfrak q)=0\;\; \textit{for some $Y$ in $\mathcal T$}\}$$
Let $G_R$ be the smallest $K$-algebraic subgroup of $G$ containing $R$. Call $n_{R,\ell}$ the greatest power of $\ell$ dividing the number of connected components of $G_R$. Call $G^j_R$ the connected component of $G_R$ containing the point $jR$. The set $\Gamma$ is infinite if and only if $\mathcal T$ contains a point in 
$$\bigcup_{j\equiv 1\!\!\!\! \pmod{n_{R,\ell}} } G^j_R[\ell^\infty](L)$$
\end{thm}

Notice that throughout the paper we replace $\ell$ by a finite set $S$ of rational primes.

To prove the existence of the densities, we apply a method by Jones and Rouse (\cite[Theorem 7]{JonesRouse}). An alternative method is due to Pink and R\"utsche, see \cite[Chapter 4]{Ruetsche}. 

To determine the conditions under which the densities are positive, we refine results of \cite{Peruccaorder} which were based on a method by Khare and Prasad (\cite[Lemma 5]{KharePrasad}). An alternative method is due to Pink, see \cite[Theorem 4.1]{Pink}.
Notice that the same method by Khare and Prasad has been applied in the following papers by Banaszak, Gajda, Kraso\'n, Bara\'nczuk and G\'ornisiewicz: \cite{BGKdetecting}, \cite{Baranczuk06}, \cite{GajdaGornisiewicz}, \cite{BanaszakKrason}.

Some explicit calculations for the density have been made by Jones and Rouse in \cite{JonesRouse}. About the order of the reductions of points on the multiplicative group and elliptic curves, see \cite{Schinzelprimitive} and \cite{CheonHahn} respectively.

A reason to study the order of the reduction of points is the following. Fix a number field $K$. Let $A$ be a simple abelian variety defined over $K$ and let $R$ be a point in $A(K)$ of infinite order. Consider the sequence $\{\ord (R \bmod\mathfrak p)\}$ indexed by the primes $\mathfrak p$ of $K$ (put $1$ if the expression is not well-defined).
This sequence determines the isomorphism class of $A$ and determines $R$ up to isomorphism. This is a corollary of the results on the support problem (\cite[Corollary 8 and Proposition 9]{Peruccatwo}).

\section{Preliminaries}\label{preliminaries}

Let $G$ be the product of an abelian variety and a torus defined over a number field $K$. Let $R$ be a point in $G(K)$. Call $G_R$ the smallest $K$-algebraic subgroup of $G$ containing $R$, which is the Zariski closure of $\mathbb ZR$. The connected component of the identity of $G_R$ is the product of an abelian variety and a torus defined over $K$ (see \cite[Proposition 5]{Peruccaorder}). Call it $G^0_R$.
Let $n_R$ be the number of connected components of $G_R$.

For every finite extension $L$ of $K$, the smallest $L$-algebraic subgroup of $G$ containing $R$ is the base change $G_R\times_K \Spec L$. Notice that $n_R$ does not depend on the field $L$ because $G^0_R$ is geometrically connected (since it has a rational point).

The point $n_RR$ is the smallest positive multiple of $R$ which belongs to $G^0_R$. 
There exists a torsion point $X$ in $G_R(\bar{K})$ of order $n_R$ such that $R-X$ belongs to $G_R^0$ (see \cite[Lemma 1]{Peruccaorder}). In particular, the point $n_RX$ is the smallest positive multiple of $X$ which belongs to $G^0_R$.
The group of connected components of $G_R$ is cyclic of order $n_R$. The connected components of $G_R$ are $G^0_R,\ldots, G^{n_R-1}_R$, where $G^i_R$ is the connected component of $G_R$ containing $iR$ (or equivalentely containing $iX$).

\begin{lem}\label{components}
For all but finitely many primes $\mathfrak p$ of $K$, the connected components of $(G_R \bmod \mathfrak p)$ are $(G^i_R \bmod \mathfrak p)$ for $i=0,\ldots, n_R-1$. In particular, the group of connected components of $(G_R \bmod \mathfrak p)$ is cyclic of order $n_R$. If $L$ is a finite Galois extension of $K$, the analogue properties hold for every prime $\mathfrak q$ of $L$ lying outside a finite set of primes of $K$ not depending on $L$.
\end{lem}
\begin{proof}
Let $F$ be a finite Galois extension of $K$ where the points in $G[n_R]$ are defined. Apply \cite[Lemma 4.4]{Kowalskikummer} to $G[n_R]$ and to $G^0_R[n_R]$. We deduce that for all but finitely many primes $\mathfrak w$ of $F$ the following holds: $(n_RX \bmod \mathfrak w)$ is the smallest positive multiple of $(X \bmod \mathfrak w)$ which belongs to $(G^0_R \bmod \mathfrak w)$.
Thus for all but finitely many primes $\mathfrak p$ of $K$ the point $(n_RR \bmod \mathfrak p)$ is the smallest positive multiple of $(R \bmod \mathfrak p)$ which belongs to $(G^0_R \bmod \mathfrak p)$. The first assertion follows.

Let $\mathfrak q$ be a prime of $L$ lying over a prime $\mathfrak p$ of $K$. 
The group of connected components of $(G_R \bmod \mathfrak q)$ is cyclic of order dividing $n_R$. Then the second assertion holds since $(G_R \bmod \mathfrak q)$ is a base change of $(G_R \bmod \mathfrak p)$, up to discarding a set of primes $\mathfrak p$ of $K$ not depending on $L$.
\end{proof}

\begin{lem}[see also {\cite[Lemma 4.4]{Kowalskikummer}}]\label{torsioniso}
Let $L$ be a finite Galois extension of $K$. Let $n$ be a positive integer such that $G[n]\subseteq G(L)$.
For every prime $\mathfrak q$ of $L$ coprime to $n$ and not lying over a finite set of primes of $K$ (not depending on $n$ nor on $L$), the reduction modulo $\mathfrak q$ gives an isomorphism from $G^i_R[n]$ to $(G^i_R \bmod \mathfrak q)[n]$ for every $i=0,\ldots, n_R-1$.
\end{lem}
\begin{proof}
By \cite[Lemma 4.4]{Kowalskikummer}, the property in the statement holds for $G^0_R[n]$ and for $G[n]$. By Lemma~\ref{components}, up to excluding a finite set of primes $\mathfrak q$ (lying over a finite set of primes of $K$ not depending on $n$ nor on $L$), we may assume that  the connected components of $(G_R \bmod \mathfrak q)$ are $(G_R^i \bmod \mathfrak q)$ for $i=0,\ldots, n_R-1$. We conclude because the reduction modulo $\mathfrak q$ maps  $G^i_R[n]$ to $(G^i_R \bmod \mathfrak q)[n]$.
\end{proof}

\begin{lem}[see also {\cite[Proposition C.1.5]{HindrySilvermanbook}}]\label{ramify}
Let $m$ be a positive integer. For every $n>0$ call $K_n$ the smallest extension of $K$ over which the $m^n$-th roots of $R$ are defined. Then the primes of $K$ which ramify in  $\bigcup_{n> 0} K_n$ are contained in a finite set.
\end{lem}
\begin{proof}
It suffices to prove that there exists a finite set $J$ of primes of $K$ (not depending on $n$) such that the following holds: every prime $\mathfrak p$ of $K$ outside this set does not ramify in $K_n$.
By \cite[Lemma 4.4]{Kowalskikummer}, there exists a finite set $J$ of primes of $K$ (not depending on $n$) satisfying the following property: for every prime $\mathfrak p$ of $K$ outside $J$ and for every prime $\mathfrak q$ of $K_n$ over $\mathfrak p$, the reduction map modulo $\mathfrak q$ is injective on $G[m^n]$.
It suffices to show that the inertia group of $\mathfrak q$ over $\mathfrak p$ is trivial. Let $\sigma$ be in the inertia group of $\mathfrak q$ over $\mathfrak p$. Then $\sigma$ induces the identity automorphism on the reduction modulo $\mathfrak q$ of the $m^n$-th roots of $R$. Because of the injectivity of the reduction modulo $\mathfrak q$ on $G[m^n]$, $\sigma$ induces the identity automorphism on the $m^n$-th roots of $R$  hence it is the identity of $\Gal(K_n/K)$.
\end{proof}

\section{On the existence of the density}\label{RJJR}

In this section we generalize a result by Jones and Rouse (\cite[Theorem 7]{JonesRouse}). We apply the same method to prove the existence of the natural density. 

The results by Pink and R\"utsche in \cite[Chapter 4]{Ruetsche} concern the existence of the Dirichlet density. Their method has the advantage (say with respect to Corollary~\ref{existencecor}) to allow the set $\mathcal T$ to be infinite.

\begin{thm}\label{RJunico}
Let $G$ be the product of an abelian variety and a torus defined over a number field $K$. Let $R$ be a point in $G(K)$. Let $S$ be a finite set of rational primes and let $m$ be the product of the elements of $S$.
Let $T$ be a point in $G[m^\infty](L)$, where $L$ is a finite Galois extension of $K$. Call $\mathcal T$ the $\Gal(\bar{K}/K)$-conjugacy class of $T$.
Then the following set of primes of $K$ has a natural density:
$$\Gamma=\{\mathfrak p : \forall \ell\in S\;\; \ord_\ell(R-T \bmod \mathfrak q)=0\;\;\textit{for some prime $\mathfrak q$ of $L$ over $\mathfrak p$}\}$$
$$=\{\mathfrak p : \forall \ell\in S\;\;\; \textit{$\forall$ prime $\mathfrak q$ of $L$ over $\mathfrak p$}\;\; \ord_\ell(R-Y \bmod \mathfrak q)=0\;\; \textit{for some $Y$ in $\mathcal T$}\}$$
\end{thm}

\begin{proof} \textit{First step.} 
For every $Y\in \mathcal T$, we have $G^0_{R-Y}=G^0_R$ because $R$ and $R-Y$ have a common multiple.
 Since $G^0_R$ and $R$ are defined over $K$, it follows that  $n_{R-Y}=n_{R-T}$ for every $Y\in \mathcal T$. If $m$ and $n_{R-T}$ are not coprime then by \cite[Proposition 2]{Peruccaorder} the set $\Gamma$ is finite and in particular it has density zero. Now assume that  $m$ and $n_{R-T}$ are coprime. By replacing $R$ and $T$ by $n_{R-T}R$ and $n_{R-T}T$ respectively, we may assume that for every $Y\in \mathcal T$ the algebraic group $G_{R-Y}$ is connected hence equal to $G^0_R$. 
Call $G'=G^0_R$, which is the product of an abelian variety and a torus defined over $K$. 

\textit{Second step.} 
Let $a$ be such that $m^a(R-Y)=m^aR$ for every $Y\in\mathcal T$. In particular, $m^aR$ belongs to $G'$.
Call $K_n$ the smallest extension of $K$ over which the $m^{n+a}$-th roots of $m^aR$ in $G'$ are defined. By Lemma~\ref{ramify}, we may consider only the primes $\mathfrak p$ of $K$ which do not ramify in $\bigcup_{n> 0} K_n$. We also avoid the primes of bad reduction. By Lemma~\ref{torsioniso}, we may also assume the following: for every $n$ and for every prime $\mathfrak w$ of $K_n$ over $\mathfrak p$ the reduction modulo $\mathfrak w$ is injective on $G'[m^{n+a}]$. Call $k_{\mathfrak w}$ the residue field. Then, for every $Y\in\mathcal T$,  the reduction modulo $\mathfrak w$ induces a bijection from the $m^n$-th roots of $R-Y$ in $G'$ to the $m^n$-th roots of $(R-Y \bmod \mathfrak w)$ in $G'_{\mathfrak w}(k_{\mathfrak w})$.

By excluding finitely many primes $\mathfrak p$ of $K$, we may also assume that $G_{\mathfrak w}$ (respectively $G'_{\mathfrak w}$) is the base change of $G_{\mathfrak p}$ (respectively $G'_{\mathfrak p}$).
In particular, we identify $G_{\mathfrak w}(k_{\mathfrak w})$ (respectively $G'_{\mathfrak w}(k_{\mathfrak w})$) with $G_{\mathfrak p}(k_{\mathfrak w})$ (respectively $G'_{\mathfrak p}(k_{\mathfrak w})$).

\textit{Third step.}
Call $H_n$ the subset of $\Gal(K_n/K)$ consisting of the automorphisms which fix some $m^n$-th root of $R-Y$ in $G'$ for some $Y\in \mathcal T$.
We write $\Frob_{\mathfrak p}$ for the Frobenius at $\mathfrak p$ without specifying the prime of $K_n$ lying over $\mathfrak p$.

Since $H_n$ is closed by conjugation, the following set of primes of $K$ is well-defined:
$$B_n=\{\mathfrak p\; :\;\; \Frob_{\mathfrak p} \in H_n\}$$
The set $B_n$ has a natural density because of the Cebotarev Density Theorem.

Now we prove that $B_n\supseteq \Gamma$ for every $n$.
Take $\mathfrak p$ in $\Gamma$ and let $\mathfrak q$ be a prime of $L$ over $\mathfrak p$. Let $Y\in \mathcal T$ be such that the order of $(R-Y \bmod \mathfrak q)$ is coprime to $m$ or equivalently such that the orbit of $(R-Y \bmod \mathfrak q)$ via the iterates of $[m]$ is periodic.
Since $(R \bmod \mathfrak q)$ belongs to $G_{\mathfrak p}(k_{\mathfrak p})$ and $(Y \bmod \mathfrak q)$ is a multiple of $(R \bmod \mathfrak q)$, the point $(R-Y \bmod \mathfrak q)$ belongs to $G_{\mathfrak p}(k_{\mathfrak p})\cap (G'(L) \bmod \mathfrak q)$.
Then $(R-Y \bmod \mathfrak q)$ has $m^n$-th roots in that set for every $n$. Fix $n$ and let $\mathfrak w$ be a prime of $K_n$ over $\mathfrak q$.
We deduce that there exists $Z$ in $G'(K_n)$ such that $m^nZ=R-Y$ and $(Z \bmod \mathfrak w)$ is in $G_{\mathfrak p}(k_{\mathfrak p})$. In particular, $Z$ is fixed by $\Frob_{\mathfrak p}$.

Now we suppose that $\mathfrak p$ belongs to $B_n$ for infinitely many $n$ and show that $\mathfrak p$ belongs to $\Gamma$ . We have to prove that for every prime $\mathfrak q$ of $L$ over $\mathfrak p$ there exists $Y\in \mathcal T$ such that the orbit of $(R-Y \bmod \mathfrak q)$ via the iterates of $[m]$ is periodic.
Since $\mathcal T$ and $G_{\mathfrak q}(k_{\mathfrak q})$ are finite sets, it suffices to show that for infinitely many $n$ the point $(R-Y \bmod \mathfrak q)$ has $m^n$-th roots in $G_{\mathfrak q}(k_{\mathfrak q})$ for some $Y\in \mathcal T$.

Let $n$ be such that $\mathfrak p$ belongs to $B_n$ and fix a prime $\mathfrak w$ of $K_n$ over $\mathfrak q$. Let $Y\in \mathcal T$ be such that there exists $Z$ in $G'(K_n)$ satisfying the following properties: $m^nZ=R-Y$ and $Z$ is fixed by $\Frob_{\mathfrak p}$. Then $(Z \bmod \mathfrak w)$ is in $G_{\mathfrak p}(k_{\mathfrak p})$ and $m^n(Z \bmod \mathfrak w)=(R-Y \bmod \mathfrak w)$. It follows that $(R-Y \bmod \mathfrak q)$ has $m^n$-th roots in $G_{\mathfrak q}(k_{\mathfrak q})$.

\textit{Fourth step.}
For every $\sigma$ in $\Gal(K_n/K)$, call $\sigma_{n}$ (respectively $\sigma_{n,\ell}$) the image of $\sigma$ in the group of automorphisms of $G'[m^{n+a}]$ (respectively  $G'[\ell^{n+a}]$). Notice that the determinant of $\sigma_{n,\ell}$ is an element of $\mathbb Z/\ell^{n+a}\mathbb Z$ and the fact that the determinant is zero is invariant by conjugation.
Then the following set of primes of $K$ is well-defined and it has a natural density because of the Cebotarev Density Theorem:
$$A_n=\{\mathfrak p\in B_n\; :\;  \det (\Frob_{\mathfrak p, n,\ell}-\id))\neq 0\; \; \forall \ell\in S \}$$

We now prove that $A_n \subseteq \Gamma$ for every $n$. It suffices to show that for every $n$ it is $A_n\subseteq A_{n+1}$ since then $A_n$ is contained in $B_n$ for infinitely many $n$.

Fix $\mathfrak p$ in $A_n$. Since $\det (\Frob_{\mathfrak p, n,\ell}-\id))\neq 0$ it follows that $\det (\Frob_{\mathfrak p, n+1,\ell}-\id))\neq 0$. Furthermore, 
the image of $(\Frob_{\mathfrak p, n,\ell}-\id)$ in $G'[\ell^{n+a}]$ has the same index as the image of $(\Frob_{\mathfrak p, n+1,\ell}-\id)$ in $G'[\ell^{n+a+1}]$. Thus the $m$-th roots of the image of $(\Frob_{\mathfrak p, n}-\id)$ belong to the image of $(\Frob_{\mathfrak p, n+1}-\id)$.

For every $Y \in \mathcal T$, let $P_Y$ be a $m^{n+1}$-th root of $R-Y$ in $G'$. Notice that any other $m^{n+1}$-th root of $R-Y$ in $G'$ differs from $P_Y$ by an element of $G'[m^{n+1}]$. Then $\Frob_{\mathfrak p}$ is in $H_{n+1}$ if and only if for some $Y\in\mathcal T$ the point $\Frob_{\mathfrak p}(P_Y)-P_Y$ is of the form $\Frob_{\mathfrak p,n+1}(X)-X$ for some $X$ in $G'[m^{n+1}]$.
Similarly, because $\mathfrak p$ is in $H_n$, we know that for some $Y$ the point $\Frob_{\mathfrak p}(mP_Y)-mP_Y$ is of the form $\Frob_{\mathfrak p,n}(X)-X$ for some $X$ in $G'[m^{n}]$. For such $Y$, the $m$-th root $\Frob_{\mathfrak p}(P_Y)-P_Y$ is of the form $\Frob_{\mathfrak p,n+1}(X)-X$ for some $X$ in $G'[m^{n+1}]$. Thus $\Frob_{\mathfrak p}$ belongs to $H_{n+1}$. We conclude that $\mathfrak p$ belongs to $A_{n+1}$.

\textit{Fifth step.} To conclude the proof, we show that the natural density of $B_n\setminus A_n$ goes to zero for $n$ going to infinity.
We have:
$$B_n\setminus A_n \subseteq \bigcup_{\ell\in S} \{\mathfrak p : \Frob_{\mathfrak p} \in H_n\; ;\;  \det (\Frob_{\mathfrak p, n, \ell}-\id))=0\}$$
Without loss of generality, we fix $\ell$ in $S$ and show that the following set (which is well-defined and whose natural density exists by the Cebotarev Density Theorem) has density going to zero for $n$ going to infinity:
$$E_n= \{\mathfrak p : \Frob_{\mathfrak p} \in H_n \;;\; \det (\Frob_{\mathfrak p, n, \ell}-\id))=0\}$$

Because of the Cebotarev Density Theorem, the density of $E_n$ is at most the maximum of 
$$ \frac{\card\{\sigma \in\Gal(K_n/K): \,\sigma_{n,\ell}=g\,;\; \sigma \in H_n\,;\; \det (g-\id))=0\}}{\card \{\sigma \in\Gal(K_n/K): \,\sigma_{n,\ell}=g\}}$$
where $g$ varies in the group of the automorphisms of $G'[\ell^{n+a}]$ induced by $\Gal(K_n/K)$.

To estimate the above ratio, we may replace $H_n$ with the subset of $\Gal(K_n/K)$ fixing some $\ell^{n+a}$-th root of $m^aR$ in $G'$.
Then we may replace $K_n$ by the smallest extension of $K$ where the $\ell^{n+a}$-th roots of $m^aR$ in $G'$ are defined (since the properties of $\sigma$ are determined by its restriction to this subfield).

By \cite[Theorem 2]{Bertrandgalois} (applied to the point $m^aR$ in $G'$) there exists a positive integer $c$, not depending on $n$ nor on $g$, such that the denominator is at least $\frac{1}{c}\card(G'[\ell^{n+a}])$.  

Now we estimate the numerator. Let $Z$ be an $\ell^{n+a}$-th root of $m^aR$ in $G'$. Any $\sigma$ such that $\sigma_{n,\ell}=g$ is determined by $\sigma(Z)-Z$.
Since $\sigma\in H_n$, $\sigma(Z)-Z$ is in the image of $g-id$. By the assumptions on $g$, the cardinality of the image of $g-id$ is at most $\frac{1}{\ell^{n+a}}\card(G'[\ell^{n+a}])$.
We deduce that the density of $E_n$ is bounded by $\frac{c}{\ell^{n+a}}$.
\end{proof}

Notice that if $R$ is a torsion point then $\Gamma$ or its complement is a finite set.

\begin{rem}\label{remRJJR1}
In Theorem~\ref{RJunico} it is not necessary to require that the point $T$ has order dividing a power of $m$. 
\end{rem}
\begin{proof}
Write $T=T'+T''$ where the order of $T'$ divides a power of $m$ and the order of $T''$ is coprime to $m$. Then $T''$ does not influence the condition defining $\Gamma$.
\end{proof}

\begin{rem}\label{remRJJR}
In the theorem, if $T=0$ we have
$$\Gamma=\{\mathfrak p : \forall \ell\in S\;\; \ord_\ell(R \bmod \mathfrak p)=0\}$$
Call $K_n$ the smallest extension of $K$ where the $m^n$-th roots of $R$ are defined.
If $G_R=G$, the density of $\Gamma$ is $$\lim_{n\rightarrow\infty} \frac{\card\{ \sigma \in \Gal(K_n/K) : \textit{$\sigma$ fixes some $m^n$-th root of $R$}\}}{\card \Gal(K_n/K)}$$
\end{rem}
\begin{proof}
In the proof of the Theorem~\ref{RJunico} (in which $a=0$, $G'=G$), notice that the density of $\Gamma$ is the limit of the density of $B_n$.
\end{proof}

\begin{cor}\label{existencecor}
Let $G$ be the product of an abelian variety and a torus defined over a number field $K$. Let $R$ be a point in $G(K)$. Let $S$ be a finite set of rational primes.
Let  $\mathcal T$ be a finite $\Gal(\bar{K}/K)$-stable subset of $G(\bar{K})_{\tors}$.  Let $L$ be a finite Galois extension of $K$ over which the points in $\mathcal T$ are defined.
Then the following set of primes of $K$ has a natural density:
$$\Gamma=\{\mathfrak p : \forall \ell\in S\;\;\; \textit{$\forall$ prime $\mathfrak q$ of $L$ over $\mathfrak p$}\;\; \ord_\ell(R-Y \bmod \mathfrak q)=0\; \textit{for some $Y$ in $\mathcal T$}\}$$
\end{cor}
\begin{proof} The set $\mathcal T$ is the disjoint union of the $\Gal(\bar{K}/K)$-orbits of its element. To each orbit we can apply Theorem~\ref{RJunico}, in view of Remark~\ref{remRJJR1}. Then $\Gamma$ is the disjoint union of finitely many sets admitting a natural density.
\end{proof}

\begin{cor}\label{strongcor}
Let $K$ be a number field and let $I=\{1,\ldots, n\}$. For every $i\in I$ let $G_i$ be the product of an abelian variety and a torus defined over $K$ and let $R_i$ be a point in $G_i(K)$. Let $S$ be a finite set of rational primes. For every $i\in I$, let $\mathcal T_i$ be 
a finite $\Gal(\bar{K}/K)$-stable subset of $G_i(\bar{K})_{\tors}$. 
Let $L$ be a finite Galois extension of $K$ where the points of $\mathcal T_i$ are defined for every $i$.
Then the following set of primes of $K$ has a natural density:
$$\Gamma=\{\mathfrak p:\; \forall \ell\; \forall i\;\; \textit{$\forall$ prime $\mathfrak q$ of $L$ over $\mathfrak p$}\;\;\;\ord_\ell(R_i-Y_i \bmod \mathfrak q)=0\; \textit{for some $Y_i$ in $\mathcal T_i$}\}$$
\end{cor}
\begin{proof}
Write $G=\prod G_i$ and $R=(R_1,\ldots, R_n)$. Call $\mathcal T$ the set of points $T=(T_1,\ldots, T_n)$ such that $T_i\in\mathcal T_i$ for every $i\in I$. Then it suffices to apply Corollary~\ref{existencecor} to $R$ and $\mathcal T$.
\end{proof}

\begin{cor}\label{exmainexistence}
Let $K$ be a number field and let $I=\{1,\ldots, n\}$. For every $i\in I$ let $G_i$ be the product of an abelian variety and a torus defined over $K$ and let $R_i$ be a point in $G_i(K)$. Let $S$ be a finite set of rational primes. For every $i\in I$ and for every $\ell \in S$, let $a_{\ell i}$ be a non-negative integer. Consider the following set of primes of $K$:
$$\Gamma=\{\mathfrak p :\; \forall \ell\in S\;\; \forall i\in I\;\; \lord(R_i \bmod \mathfrak p)=a_{\ell i}\}$$
The set $\Gamma$ has a natural density.
\end{cor}
\begin{proof}
Call $m$ the product of the elements of $S$. 
For every $i$, let $\mathcal T_i$ be the set consisting of the points $Y_i$ in $G_i[m^\infty](\bar{K})$ satisfying $\lord(Y_i)=a_{\ell i}$ for every $\ell \in S$. Let $L$ be a finite Galois extension of $K$ where the points of $\mathcal T_i$ are defined for every $i$. 
It suffices to apply Corollary~\ref{strongcor} since by Lemma~\ref{torsioniso}, up to excluding finitely many primes $\mathfrak p$, we have 
$$\Gamma=\{\mathfrak p:\; \forall \ell\; \forall i\;\; \textit{$\forall$ prime $\mathfrak q$ of $L$ over $\mathfrak p$}\;\;\;\ord_\ell(R_i-Y_i \bmod \mathfrak q)=0\; \textit{for some $Y_i$ in $\mathcal T_i$}\}$$
\end{proof}

\section{On the positivity of the density}

Theorems~\ref{exmainl}~and~\ref{thmintro} are proven respectively in Theorems~\ref{exmain}~and~\ref{tuttolemma}.

\begin{thm}\label{tuttolemma}
Let $G$ be the product of an abelian variety and a torus defined over a number field $K$. Let $R$ be a point in $G(K)$. Let $S$ be a finite set of rational primes. Call $m$ the product of the elements of $S$.
Let $L$ be a finite Galois extension of $K$. Let $\mathcal T$ be a $\Gal(\bar{K}/K)$-stable subset of $G[m^\infty](L)$.
Then the following set of primes of $K$ is either finite or it has a positive natural density:
$$\Gamma=\{\mathfrak p : \forall \ell\in S\;\;\; \textit{$\forall$ prime $\mathfrak q$ of $L$ over $\mathfrak p$}\;\; \ord_\ell(R-Y \bmod \mathfrak q)=0\; \textit{for some $Y$ in $\mathcal T$}\}$$
Let $G_R$ be the smallest $K$-algebraic subgroup of $G$ containing $R$. For every $\ell$, call $n_{R,\ell}$ the greatest power of $\ell$ dividing the number of connected components of $G_R$. Call $G^j_R$ the connected component of $G_R$ containing $jR$.
The set $\Gamma$ is infinite if and only if the set $\mathcal T$ contains a point which can be written as the sum for $\ell\in S$ of elements in
$$\bigcup_{j\equiv 1\!\!\!\! \pmod{n_{R,\ell}}} G^j_R[\ell^\infty](L)$$
\end{thm}
\begin{proof}
The existence of the density was proven in Corollary~\ref{existencecor}.
Since the set $\Gamma$ increases by enlarging $\mathcal T$, we may reduce to the case where $\mathcal T$ is the $\Gal(\bar{K}/K)$-orbit of a point $T$. 
By \cite[Main Theorem]{Peruccaorder} applied to the point $R-T$, the set $\Gamma$ is infinite if and only if $n_{R-T}$ is coprime to $m$.

Suppose that $\Gamma$ is infinite. By \cite[Theorem 7]{Peruccaorder} applied to the point $n_{R-T}(R-T)$, there exists a positive density of primes $\mathfrak p$ of $K$ such that for some prime $\mathfrak q$ of $L$ over $\mathfrak p$ it is $\lord(R-T \bmod \mathfrak q)=\lord(n_{R-T}(R-T) \bmod \mathfrak q)=0$ for every $\ell\in S$. Hence $\Gamma$ has a positive density.

Write $T=\sum_\ell T_\ell$ where $T_\ell$ is in  $G[\ell^\infty](L)$. Notice that $T_\ell$ is a multiple of $T$ for every $\ell\in S$.
If $\Gamma$ is infinite, there exist infinitely many primes $\mathfrak q$ of $L$ such that $\lord(R-T \bmod \mathfrak q)=0$. For every $\ell\in S$ the point $(T_\ell \bmod \mathfrak q)$ is a multiple of $(R \bmod \mathfrak q)$ hence it belongs to $(G_R \bmod \mathfrak q)$. By applying Lemma~\ref{torsioniso} to $G$ and $G_R$, we deduce that $T_\ell$ belongs to $G_R$ for every $\ell \in S$.
Then to prove the criterion in the statement we may assume that the point $T$ is such that $T_\ell$ belongs to $G_R$ for every $\ell \in S$.

Notice that  $n_{R-T}$ is coprime to $\ell$ if and only if $n_{R-T_{\ell}}$ is coprime to $\ell$. 
To conclude, we show that $n_{R-T_{\ell}}$ is coprime to $\ell$ if and only if the point $T_\ell$ belongs to $G^j_R[\ell^\infty](L)$ for some $j\equiv 1\!\! \pmod{n_{R,\ell}}$. The last condition is equivalent to saying that $R-T_{\ell}$ belongs to $G^j_R[\ell^\infty](L)$ for some $j\equiv 0 \!\! \pmod{n_{R,\ell}}$.

Let $R-T_{\ell}$ belong to $G^j_R$ and let $X$ be as in Section~\ref{preliminaries}. Then $G^j_R=G^0_R+jX$ and the smallest multiple of $jX$ lying in $G^0_R$ is $[n_R/(n_R,j)] jX$. Since $G^0_{R-T_\ell}=G^0_R$, we deduce that $n_{R-T_{\ell}}$ is coprime to $\ell$ if and only if $n_R/(n_R,j)$ is coprime to $\ell$. This is equivalent to saying that $j\equiv 0\!\! \pmod{n_{R,\ell}} $.
\end{proof}

\begin{cor}\label{tutto}
Let $K$ be a number field, let $I=\{1,\ldots, n\}$. For every $i \in I$, let $G_i$ be the product of an abelian variety and a torus defined over $K$ and let $R_i$ be a point in $G_i(K)$. Let $S$ be a finite set of rational primes. Call $m$ the product of the elements of $S$.
Let $L$ be a finite Galois extension of $K$. For every $i$, let $\mathcal T_i$ be a $\Gal(\bar{K}/K)$-stable subset of $G_i[m^\infty](L)$.
Then the following set of primes of $K$ is either finite or it has a positive natural density:
$$\Gamma=\{\mathfrak p :  \forall i \;\; \forall \ell\;\; \textit{$\forall$ prime $\mathfrak q$ of $L$ over $\mathfrak p$}\;\;\; \ord_\ell(R_i-Y_i \bmod \mathfrak q)=0\; \textit{for some $Y_i$ in $\mathcal T_i$}\}$$
Write $G=\prod_{i=1}^n G_i$ and $R=(R_1,\ldots, R_n)$. Let $G_R$ be the smallest $K$-algebraic subgroup of $G$ containing $R$. For every $\ell$, call $n_{R,\ell}$ the greatest power of $\ell$ dividing the number of connected components of $G_R$. Call $G^j_R$ the connected component of $G_R$ containing $jR$. Let $\mathcal T$ be the product of the $\mathcal T_{i}$ for $i\in I$. 
The set $\Gamma$ is infinite if and only if the set $\mathcal T$ contains a point which can be written as the sum for $\ell\in S$ of elements in
$$\bigcup_{j\equiv 1\!\!\!\! \pmod{n_{R,\ell}}} G^j_R[\ell^\infty](L)$$
\end{cor}
\begin{proof}
Notice that 
$$\Gamma=\{\mathfrak p : \; \forall \ell\in S\;\; \textit{$\forall$ prime $\mathfrak q$ of $L$ over $\mathfrak p$}\;\;\; \ord_\ell(R-Y \bmod \mathfrak q)=0\; \textit{for some $Y$ in $\mathcal T$}\}$$
Then it suffices to apply Theorem~\ref{tuttolemma}.
\end{proof}

\begin{thm}\label{exmain}
Let $K$ be a number field, let $I=\{1,\ldots, n\}$. For every $i \in I$, let $G_i$ be the product of an abelian variety and a torus defined over $K$ and let $R_i$ be a point in $G_i(K)$. Let $S$ be a finite set of rational primes. 
For every $i\in I$ and for every $\ell \in S$, let $a_{\ell i}$ be a non-negative integer. Consider the following set of primes of $K$:
$$\Gamma=\{\mathfrak p :\; \forall \ell\in S\;\; \forall i\in I\;\; \lord(R_i \bmod \mathfrak p)=a_{\ell i}\}$$

The set $\Gamma$ is either finite or it has a positive natural density.

Write $G=\prod_{i=1}^n G_i$ and $R=(R_1,\ldots, R_n)$. Let $G_R$ be the smallest $K$-algebraic subgroup of $G$ containing $R$. 
For every $\ell$, call $n_{R,\ell}$ the greatest power of $\ell$ dividing the number of connected components of $G_R$. Call $G^j_R$ the connected component of $G_R$ containing $jR$.

The set $\Gamma$ is infinite if and only if one of the following equivalent conditions is satisfied:

\begin{description}
\item[(i)] for every $\ell \in S$ there exists a torsion point $T_\ell=(T_{\ell 1},\ldots, T_{\ell n})$ such that $\lord (T_{\ell i})=a_{\ell i}$ for every $i\in I$ and $T_{\ell}$ belongs to $$\bigcup_{j\equiv 1\!\!\!\! \pmod{n_{R,\ell}} } G^j_R[\ell^\infty]$$
\item[(ii)] for every $\ell\in S$ there exists a torsion point $T_\ell=(T_{\ell 1},\ldots, T_{\ell n})$ in $G^1_R(\bar{K})$ such that $\lord(T_{\ell i})=a_{\ell i}$ for every $i\in I$.
\end{description}

\end{thm}

\begin{lem}\label{puntopunti}
In Theorem~\ref{exmain}, suppose that condition (ii) is satisfied. Then there exists a torsion point $T=(T_{1},\ldots, T_{n})$ in $G^1_R(\bar{K})$ such that $\lord(T_{i})=a_{\ell i}$ for every $i\in I$ and for every $\ell\in S$.
\end{lem}
\begin{proof}
For every $\ell\in S$, the torsion point $T_\ell-X$ belongs to $G^0_R(\bar{K})$. Then we can write $T_\ell-X=Z_\ell+Z'_\ell$, where $Z_\ell$ is a point in  $G^0_R[\ell^{\infty}]$ and $Z'_\ell$ is a torsion point in $G^0_R(\bar{K})$ of order coprime to $\ell$.
Define $T=\sum_\ell Z_\ell + X$. The point $T$ is a torsion point in $G^1_R(\bar{K})$. 
For every $\ell\in S$ and for every $i\in I$ we have:
$$\lord(T_{i})=\lord(\sum_\ell Z_{\ell i} + X_i)=\lord(Z_{\ell i} + X_i)=\lord(Z_{\ell i}+Z'_{\ell i} + X_i)=\lord(T_{\ell i})=a_{\ell i}$$ 
\end{proof}

\begin{proof}[Proof of Theorem~\ref{exmain}]
The existence of the density for $\Gamma$ was proven in Corollary~\ref{exmainexistence}.

Call $m$ the product of the elements of $S$.
Let $L$ be a finite Galois extension of $K$ where the points in $G_i[\ell^{a_{\ell i}}]$ are defined for every $\ell\in S$ and for every $i\in I$.
We may assume (see Lemma~\ref{torsioniso}) that for every prime $\mathfrak q$ of $L$ the reduction modulo $\mathfrak q$ gives a bijection from $G_i[\ell^{a_{\ell i}}]$ to $(G_i[\ell^{a_{\ell i}}] \bmod \mathfrak q)$, for every $\ell\in S$ and for every $i\in I$.

Let $\mathcal T$ be the set consisting of the points $Y=(Y_1,\ldots, Y_n)$ in $G[m^\infty]$ such that $\lord(Y_i)=a_{\ell i}$ for every $\ell\in S$ and for every $i\in I$. Notice that $\mathcal T$ is contained in $G[m^\infty](L)$ and it is $\Gal(\bar{K}/K)$-stable.
A prime $\mathfrak p$ of $K$ belongs to $\Gamma$ if and only if for every prime $\mathfrak q$ of $L$ over $\mathfrak p$ the following holds: for some $Y\in \mathcal T$ $\lord(R-Y \bmod \mathfrak q)=0$ for every $\ell\in S$.
Apply Theorem~\ref{tuttolemma} to $R$ and $\mathcal T$. We deduce that the set $\Gamma$ is infinite if and only if it has a positive density. We also deduce that $\Gamma$ is infinite if and only if $\mathcal T$ contains a point $T=(T_1,\ldots, T_n)$ with the following property: we can write $T=\sum_\ell T_\ell$ where for every $\ell\in S$ the point $T_\ell$ is in $G^j_R[\ell^{\infty}](L)$ for some $j\equiv  1 \!\!\pmod{n_{R,\ell}}$. Notice that $\mathcal T$ contains such an element if and only if condition (i) is satisfied.

Suppose again that $\Gamma$ is infinite. We show that condition (ii) is satisfied. Without loss of generality, fix $\ell\in S$. Because of condition (i) there exists $T_\ell=(T_{\ell 1},\ldots, T_{\ell n})$  such that $\lord(T_{\ell i})=a_{\ell i}$ for every $i\in I$ in $G^j_R[\ell^\infty](L)$ for some $j\equiv  1 \!\!\pmod{n_{R,\ell}}$. Let $X$ be as in section~\ref{preliminaries} and notice that the order of $(j-1)X$ is coprime to $\ell$. Since $G^j_R(\bar{K})=G^1_R(\bar{K})+(j-1)X$ we deduce that $T_\ell -(j-1)X$ is in $G^1_R(\bar{K})$ and satisfies the properties of condition (ii).

Viceversa, suppose that condition (ii) is satisfied. By Lemma~\ref{puntopunti}, there exists a torsion point $T=(T_{1},\ldots, T_{n})$ in $G^1_R(\bar{K})$ such that $\lord(T_{i})=a_{\ell i}$ for every $i\in I$ and for every $\ell\in S$. In particular, the point $R-T$ belongs to $G^0_R(\bar{K})$. Furthermore, $G^0_{R-T}=G^0_R$ since $R$ and $R-T$ have a common multiple. We deduce that $G_{R-T}$ is connected.

 Let $F$ be a finite Galois extension of $K$ where $T$ is defined. By applying \cite[Theorem 7]{Peruccaorder} to the point $R-T$, we find infinitely many primes $\mathfrak p$ of $K$ such that for some prime $\mathfrak w$ of $F$ over $\mathfrak p$ it is $\lord(R-T \bmod \mathfrak w)=0$ for every $\ell\in S$.
 
Up to excluding finitely many primes $\mathfrak p$, we may assume that the order of $(T_i \bmod \mathfrak w)$ equals the order of $T_i$ for every $i\in I$.
 
Then such primes $\mathfrak p$ belong to $\Gamma$ since for every $\ell\in S$ and for every $i\in I$ it is
$$\lord(R_i \bmod \mathfrak p)=\lord(R_i \bmod \mathfrak w)=\lord(T_i \bmod \mathfrak w)=\lord T_i=a_{\ell i}$$
\end{proof}

Suppose that in Theorem~\ref{exmain} every $G_i$ and every $R_i$ is non-zero. Then the condition $G_R=G$ implies that for every choice of the parameters $a_{\ell i}$ the set $\Gamma$ is infinite. The condition $G_R=G$ is equivalent to saying that $R$ generates a free $\End_K G$-submodule of $G(K)$, see \cite[Remark 6]{Peruccaorder}.
The following example shows that the set $\Gamma$ may be infinite for every choice of the parameters even if  $G_R\neq G$.

\begin{exa}
Let $E$ be an elliptic curve over $\mathbb Q$ without complex multiplication and such that $E(\mathbb Q)$ contains three points $P_1$, $P_2$ and $P_3$ which are $\mathbb Z$-linearly independent. For example consider the curve $[0,0,1,-7,6]$ of \cite{Cremona}.
Let $I=\{1,2\}$ and let $S=\{\ell\}$. Let $G_1=G_2=E^2$. Consider the points $R_1=(P_1,P_3)$ and $R_2=(P_2,P_3)$. Let $a_1$ and $a_2$ be non-negative integers. There exist infinitely many primes $\mathfrak p$ such that $\lord(R_i \bmod \mathfrak p)=a_i$ for $i=1,2$. Indeed, the point $(P_1,P_2,P_3)$ is independent in $E^3$ so we can apply \cite[Proposition 12]{Peruccaorder}. Thus we find infinitely many $\mathfrak p$ such that $\lord(P_i \bmod \mathfrak p)=a_i$ for $i=1,2$ and $\lord(P_3 \bmod \mathfrak p)=0$.
\end{exa}

\begin{rem}
Suppose that the number of connected components of $G_R$ is coprime to $\ell$. Then in condition (ii) of Theorem~\ref{exmain} it suffices to require that $T_\ell$ is in $G_R$ and not necessarily in $G^1_R$.
In general, it suffices to require that  $T_\ell$ is in $G^b_R$ for some $b$ coprime to $\ell$.
\end{rem}
\begin{proof}
Let $X$ be as in section~\ref{preliminaries}. If the number of connected components of $G_R$ is coprime to $\ell$ then the order of $X$ is coprime to $\ell$. Then by summing to $T_\ell$ a multiple of $X$ we may assume that $T_\ell$ is in $G^1_R$.
For the second assertion, notice that $G^b_R=G^1_{bR}$. So by applying Theorem~\ref{exmain} to the point $bR$ we find infinitely many primes $\mathfrak p$ of $K$ such that for every $i\in I$ and for every $\ell\in S$ it is
$$\lord(R_i \bmod \mathfrak p)=\lord(bR_i \bmod \mathfrak p)=a_{\ell i}$$
We deduce that the set $\Gamma$ is infinite. 
\end{proof}

\begin{rem}\label{remgammal}
With the notations of Theorem~\ref{exmain}, for every $\ell\in S$ define the following set: 
$$\Gamma_\ell=\{\mathfrak p \in K :\;\; \forall i\in I\;\; \lord(R_i \bmod \mathfrak p)=a_{\ell i}\}$$
We have  $\Gamma=\cap_\ell \Gamma_\ell$ and $\Gamma$ is an infinite set if and only if $\Gamma_\ell$ is an infinite set for every $\ell\in S$.
\end{rem}
\begin{proof}
In Theorem~\ref{exmain}, condition (ii) is a collection of conditions for every $\ell\in S$.
\end{proof}

For one point of infinite order we have:

\begin{cor}\label{onepoint}
Let $G$ be the product of an abelian variety and a torus defined over a number field $K$. Let $R$ be a point in $G(K)$ of infinite order. Let $S$ be a finite set of rational primes. For every $\ell\in S$ let $a_\ell$ be a non-negative integer.
Consider the following set of primes of $K$:
$$\Gamma=\{\mathfrak p :\; \forall \ell\in S\;\; \lord(R \bmod \mathfrak p)=a_{\ell}\}$$
The set $\Gamma$ is either finite or it has a positive natural density.
Let $G_R$ be the smallest $K$-algebraic subgroup of $G$ containing $R$ and call $n_R$ the number of connected components of $G_R$. Then $\Gamma$ is infinite if and only if for every $\ell$ in $S$ it is $a_\ell\geq v_\ell(n_R)$. Furthermore, $n_R$ is the greatest positive integer dividing the order of $(R \bmod \mathfrak p)$ for all but finitely many primes $\mathfrak p$ of $K$.
\end{cor}
\begin{proof}
The assertions are consequences of \cite[Main Theorem]{Peruccaorder} and Corollary~\ref{exmainexistence}. 
\end{proof}

Notice that $G^1_R(\bar{K})$ contains a torsion point of order $n$ if and only if $n$ is a multiple of $n_R$.  This follows from the fact that $G^1_R(\bar{K})= X+G^0_R(\bar{K})$, where $X$ is as in Section~\ref{preliminaries}.

\section*{Acknowledgements}I thank Rafe Jones and Jeremy Rouse for helpful discussions. I thank Peter Jossen, Emmanuel Kowalski and Dino Lorenzini for useful comments.


\vspace{0.5cm}

Antonella Perucca

EPFL Station 8, CH-1015, Lausanne, Switzerland  
\vspace{0.2cm}

antonella.perucca@epfl.ch
\end{document}